\documentclass[oneside,english]{amsart}
\usepackage[T1]{fontenc}
\usepackage[latin9]{inputenc}
\usepackage{amsthm}
\usepackage{amstext}
\usepackage{amssymb}
\usepackage{esint}
\usepackage{amscd}
\makeatletter
\numberwithin{equation}{section}
\numberwithin{figure}{section}
\theoremstyle{plain}
\newtheorem{thm}{\protect\theoremname}[section]
\newtheorem{prop}[thm]{\protect\propositionname}
\theoremstyle{plain}
\theoremstyle{definition}

\theoremstyle{plain}

\newtheorem{cor}[thm]{\protect\corollaryname}
\theoremstyle{plain}
\newtheorem{rem}[thm]{\protect\remarkname}
\theoremstyle{plain}

\theoremstyle{plain}
\makeatother

\usepackage{babel}
\providecommand{\definitionname}{Definition}
\providecommand{\lemmaname}{Lemma}
\providecommand{\assumptionname}{Assumption}
\providecommand{\theoremname}{Theorem}
\providecommand{\corollaryname}{Corollary}
\providecommand{\remarkname}{Remark}
\providecommand{\propositionname}{Proposition}
	
\DeclareMathOperator{\loc}{loc}

\DeclareMathOperator{\esssup}{esssup}

\DeclareMathOperator{\cp}{cap}

\DeclareMathOperator{\divg}{div}
\DeclareMathOperator{\med}{med}

\begin{document}

\title[The Laplacian with density and sharp Sobolev-Orlicz embeddings]{Eigenvalues of the Neumann Laplacian with density and sharp Sobolev-Orlicz embeddings}

\author{Alexander Menovschikov}
\begin{abstract}
    We provide the estimates for the constant in the weighted Poincar\'e inequality for a special class of planar domains and weights. Based on this, we prove the lower bounds for the first non-zero eigenvalue $\mu_\rho$ of the Neumann Laplacian with density $\rho$. These estimates depend on the density function and the geometry of the domain. In particular, it is shown, that $\mu_\rho$ can be made arbitrarily large by changing the mass density of the domain.
\end{abstract}
\maketitle

\section{Introduction}

Let $\Omega \subset \mathbb{R}^2$ be a bounded domain. We consider the Neumann eigenvalues problem for the Laplace operator with density.
    \begin{align}\label{DensLaplacian}
        \begin{cases}
            -\Delta u=\mu_\rho \rho u \,\, &\text{ in } \Omega,\\
            \frac{\partial u}{\partial\nu}=0 &\text{ on } \partial\Omega.
        \end{cases}
    \end{align}
Here $\rho$ is a positive continuous function belonging to $L^1(\Omega)$. This function characterizes the mass density of the membrane $\Omega$.

It is known (e.g. \cite{He06}) that if the domain $\Omega$ and the density $\rho$ are such that the embedding $W^{1,2}(\Omega) \hookrightarrow L^2(\Omega, \rho)$ is compact, then for problem \eqref{DensLaplacian} there exists a sequence of non-negative eigenvalues that goes to $+\infty$
    $$
        0 = \mu^0_{\rho} \leq \mu^1_{\rho} \leq \mu^2_{\rho} \leq \dots,
    $$
and according to the Min-Max principle, the first non-zero Neumann eigenvalue $\mu^1_{\rho} = \mu_{\rho}$ is given by
    $$
        \mu_{\rho} = \inf\limits_{u \in W^{1,2}(\Omega), u \ne 0, \int_{\Omega}u\rho = 0} \frac{\int_{\Omega}|\nabla u(x)|^2 \, dx}{\int_{\Omega}|u(x)|^2 \rho(x) \, dx},
    $$
or, equivalently 
    \begin{equation}\label{MinMax}
        \mu_\rho (\Omega) = \frac{1}{(B^\rho_{2,2}(\Omega))^{2}},
    \end{equation}
where $B^\rho_{2,2}(\Omega)$ is the best constant in the following Poincar\'e inequality:
    \begin{equation}\label{Poincare_22}
        \inf\limits_{c \in \mathbb{R}}\|u-c\|_{L^2(\Omega, \rho)} \leq B^\rho_{2,2}(\Omega) \|\nabla u\|_{L^2(\Omega)}.
    \end{equation}

The question of the dependence of the eigenvalues of problem \eqref{DensLaplacian} (or the corresponding problem in a weak formulation) on the density has been intensively investigated in recent years (see, e.g, \cite{He06, CSS15, CS19} and references therein). A summary of other problems that involve non-homogenuity of the mass density can be found in \cite{Ch18} and the references therein. Usually, to ensure compactness of the embedding $W^{1,2}(\Omega) \hookrightarrow L^2(\Omega, \rho)$ it is assumed that $\Omega$ is a Lipschitz domain and the density $\rho$ is in $L^\infty(\Omega)$.

The main goal of our paper is to provide more general conditions on the domain $\Omega$ and the density $\rho$ under which the Poincar\'e inequality \eqref{Poincare_22} holds and the embedding $W^{1,2}(\Omega) \hookrightarrow L^2(\Omega, \rho)$ is compact. In it turns it means that the spectrum of the Laplacian is discrete and Min-Max Principle for the eigenvalues can be applied. At the same time it implies that we can define the firs eigenvalue $\mu_{\rho}(\Omega)$ by the equality \eqref{MinMax}. Based on this, we prove the lower estimates for $\mu_{\rho}(\Omega)$ depending on the density $\rho$ and the geometry of the domain $\Omega$ only.

Our considerations are based on the approach first introduced in \cite{GG94} and then applied to the Sobolev embeddings theory and to spectral problems, for example, in \cite{GU09, GU16, GU16_1}. The idea is to prove the compactness of the embedding $W^{1,2}(\Omega) \hookrightarrow L^2(\Omega, \rho)$ in the following way. Let us consider a domain $\Omega'$ supporting the compact non-weighted embedding $W^{1,p}(\Omega') \hookrightarrow L^q(\Omega')$, $1 \leq p \leq 2$. Further we take a homeomorphism $\varphi: \Omega' \to \Omega$. With the help of a bounded composition operator $\varphi^\ast$ generating by $\varphi$, we shift functions from $\Omega$ to $\Omega'$. Then we use the compact non-weighted embedding on $\Omega'$ and go back to the domain $\Omega$ using another composition operator $(\varphi^{-1})^\ast$ generating by the inverse mapping $\varphi^{-1}: \Omega \to \Omega'$. It can be described by the following anticommutative diagram:
    \begin{equation}\label{Diag}
        \begin{CD}
            W^{1,2}(\Omega) @>{\varphi^\ast}>> W^{1,p}(\Omega') \\
            @VVV @VVV \\
            L^2(\Omega, \rho) @<{(\varphi^{-1})^\ast}<< L^q(\Omega')
        \end{CD}
    \end{equation}

If both $\varphi^\ast$ and $(\varphi^{-1})^\ast$ are bounded and the embedding operator $i: W^{1,p}(\Omega') \to L^q(\Omega')$ is compact, then the embedding $W^{1,2}(\Omega) \hookrightarrow L^2(\Omega, \rho)$ is also compact.

Since we are considering the Laplace operator on the plane, it is natural to use conformal mappings in the approach described above. Conformal mappings generate an isometry of Sobolev spaces $L^{1,2}$ and by the Riemann Mapping Theorem such a mapping always exists, and we can consider the unit disc $\mathbb{D}(0,1)$ as a target domain $\Omega'$ if $\Omega$ is simply connected.

Following this approach, in Section 3 we prove the Poincar\'e inequality \eqref{Poincare_22} and establish a connection between the first eigenvalue $\mu_\rho(\Omega)$ of the Laplacian with density on $\Omega$ and the first eigenvalue $\mu_{p,q}(\mathbb{D})$ of the homogeneous $(p,q)$-Laplacian on the unit disk. Namely, the following theorem holds (the domain $\Omega$ is $\alpha$-regular, if $\varphi' \in L^\alpha(\mathbb{D})$ for $\varphi: \mathbb{D} \to \Omega$, see Section 2 for details):

\begin{thm}
    Let $\Omega \subset \mathbb{R}^2$ be a simply connected bounded domain. Let also $\rho$ be a density and $1 \leq p < 2$, $2 < q \leq \frac{2p}{2-p}$. Consider a conformal mapping $\varphi: \mathbb{D} \to \Omega$. If
    $$
        K_q(\Omega, \rho, \varphi) = \Big\| \frac{\rho}{J^{2/q}_{\varphi^{-1}}} \Big\|_{L^{\frac{q}{q-2}}(\Omega)} < \infty,
    $$
    then Sobolev-Poincar\'e inequality \eqref{Poincare_22} holds with the constant 
    $$
        B^\rho_{2,2}(\Omega) \leq \pi^{\frac{2-p}{p}} B_{q,p}(\mathbb{D}) \sqrt{K_q(\Omega, \rho, \varphi)}.
    $$

    If, in addition, $\Omega$ is an $\alpha$-regular domain with $\alpha \in (\frac{4}{p}, \infty]$, and $q < \frac{2p}{2-p}$, then the embedding 
    $W^{1,2}(\Omega) \hookrightarrow L^2(\Omega, \rho)$ is compact and 
    $$
        \mu_{\rho}(\Omega) \geq \frac{\mu_{p,q}(\mathbb{D})}{2^{2p}\pi^{\frac{2(2-p)}{p}} K_q(\Omega, \rho, \varphi)}.
    $$
\end{thm}

Under additional assumptions on $\Omega$ (when $\Omega$ is a $K$-quasidisc), we can prove a lower bound for $\mu_{\rho}(\Omega)$ that does not involve the Jacobian of a conformal mapping. Moreover, it is shown that by changing the mass density, we can make $\mu_{\rho}(\Omega)$ as large as we wish.

The main results of the article contains in Section 4. These results are a refined version of the propositions of Section 3 in the sense of an optimal Sobolev embedding of the space $W^{1,2}(\mathbb{D})$ into the Orlicz space of exponential type. But, unfortunately, due to the use of Orlicz settings, the conditions in this case are much more involved.

\begin{thm}
    Let $\Omega \subset \mathbb{R}^2$ be a simply connected bounded domain. Let also $\rho$ be a density and Young functions $M$ and $\Phi$ be defined as $M(u) = e^{u^2}-1$ and $\Phi(u) = u\log(u+e)$. Consider a conformal mapping $\varphi: \mathbb{D} \to \Omega$. If
    $$
        K_{\Phi}(\Omega, \rho, \varphi)  = \Big\|\frac{\rho}{J_{\varphi^{-1}}\Phi^{-1}(\frac{1}{J_{\varphi^{-1}}})}\Big\|_{L^\Phi(\Omega)} < \infty,
    $$
    then Sobolev-Poincar\'e inequality \eqref{Poincare_22} holds with the constant
    $$
        B^\rho_{2,2}(\Omega) \leq 2\sqrt{3} B_{M,2}(\mathbb{D}) \sqrt{K_{\Phi}(\Omega, \rho, \varphi)}.
    $$
    
    If, in addition, $\Omega$ is an $\alpha$-regular domain with $\alpha \in (2, \infty]$, and $M_\varepsilon \prec M$ ($\Phi \prec \Phi_\varepsilon$), then the embedding 
    $W^{1,2}(\Omega) \hookrightarrow L^2(\Omega, \rho)$ is compact and 
    $$
        \mu_{\rho}(\Omega) \geq \frac{1}{18 (B_{M_\varepsilon,2}(\mathbb{D}))^2 K_{\Phi_\varepsilon}(\Omega, \rho, \varphi)}.
    $$
\end{thm}

As in $L^p$-case, when $\Omega$ is a $K$-quasidisk, a refined version of the lower estimate independent on $J_{\varphi^{-1}}$ is also proved in Orlicz settings.

\section{Preliminaries}

\subsection{Convex functions}

Here we recall the basic facts of the theory of Young functions, which underlies the Orlicz space theory. For detailed theories of Young functions and Orlicz spaces, as well as proofs, see, for example, \cite{KR, RR}.

A non-decreasing left-continuous convex function $M: [0,\infty) \to [0,\infty]$ is a \textit{Young function} if
    $$
        M(0) = \lim_{u \to 0^+} M(u) = 0 \quad \text{and} \quad \lim_{u \to \infty} M(u) = \infty.
    $$

Given a Young function $M$, the \textit{complementary function} $M^*: [0,\infty) \to [0,\infty]$ is defined by
    $$
        M^*(v) = \sup\{uv - M(u) : u \geq 0\}.
    $$
The complementary function to a Young function is also a Young function and $(M^*)^* = M$. For $M$ and $M^*$ the \textit{Young inequality} holds:
    \begin{equation}\label{YoungIneq}
        uv \leq M(u) + M^*(v).
    \end{equation}

The crucial role in the theory of Orlicz spaces is played by growth conditions for Young functions. We will say that a function $M_1$ has an essentially greater growth and will write
    $$
        M_1 \prec M_2
    $$
if for all $k>0$
    $$
        \lim\limits_{u \to \infty}\frac{M_1(ku)}{M_2(u)} = 0.
    $$
    
Two Young functions are \textit{equivalent} ($M_1 \sim M_2$) if there exist positive constants $u_1$, $u_2$ and $k_1$, $k_2$ such that
    $$
        M_1(u) \leq M_2(k_1u) \text{ for all } u \geq u_1 \text{ and } M_2(u) \leq M_1(k_2u) \text{ for all } u \geq u_2.
    $$
It is important property for us, since two equivalent Young functions define the same Orlicz space.

In this paper we will use the following two conditions for the growth of Young functions:

A Young function $M$ satisfies the \textit{$\Delta'$-condition (globally)}, $M \in \Delta'$, if there exist constants $C_{\Delta'}> 0$ and $u_0>0$ such that 
    \begin{equation}
        M(uv) \leq C_{\Delta'} M(u) M(v) \text{ for all } u, v \geq u_0 \,\, (u,v \geq 0).
    \end{equation}
If a Young function satisfies the $\Delta'$-condition then it grows not faster then power functions.

A Young function $M$ satisfies the \textit{$\nabla'$-condition (globally)}, $M \in \nabla'$, if there exist constants $C_{\nabla'}>0$ and $u_0>0$ such that 
    \begin{equation}
        M(C_{\nabla'} uv) \geq  M(u) M(v) \text{ for all } u, v \geq u_0 \,\, (u,v \geq 0).
    \end{equation}
Young functions with the $\nabla'$-condition grow faster than any power function. Moreover, if $M \in \Delta'$, then $M^* \in \nabla'$ and vice versa.

\subsection{Weighted Lebesgue and Orlicz spaces}

Let $\Omega$ be an open subset of $\mathbb{R}^2$. We will call a function $\rho: \Omega \to (0,\infty)$ a \textit{density}, if $\rho$ is continuous and is in $L^1(\Omega)$.

As usual, the \textit{weighted Lebesgue space} $L^p(\Omega, \rho)$, $1 \leq p \leq \infty$ with a density $\rho$ is a set of all (equivalence classes of) locally integrable functions $f: \Omega \to \mathbb{R}$ such that
    $$
        \|f\|_{L^p(\Omega, \rho)} = \left( \int_\Omega |f(x)|^p \rho(x) \, dx\right)^{\frac{1}{p}} < \infty.
    $$

Now, let $M$ be a Young function. The \textit{Orlicz space} $L^M(\Omega)$ is the set of all (equivalence classes of) locally integrable functions $f: \Omega \to \mathbb{R}$ such that
    $$
        \int_\Omega M(\alpha|f(x)|) \, dx < \infty
    $$
for some $\alpha>0$. When $L^M(\Omega)$ equipped with the \textit{Luxemburg norm}
    \begin{equation}
        \|f\|_{L^{M}(\Omega)}= \inf\{\lambda \in (0,\infty) : \int_\Omega M(\frac{f(x)}{\lambda}) \, dx \leq 1 \}
    \end{equation}
or with the \textit{Orlicz norm}
    \begin{equation}
        \|f\|_{L^{(M)}(\Omega)} = \sup \{ \|fg\|_{L^1(\Omega)} :  \int_\Omega M^*(g(x)) \, dx \leq 1\} = \sup \{ \|fg\|_{L^1(\Omega)} :  \|g\|_{L^{M^\ast}(\Omega)} \leq 1\}
    \end{equation}
it becomes a Banach space. Moreover,
    $$
        \|f\|_{L^{M}(\Omega)} \leq \|f\|_{L^{(M)}(\Omega)} \leq 2\|f\|_{L^{M}(\Omega)}.
    $$

The Luxemburg norm has the following property
    \begin{equation}\label{NormProp}
        \int_{\Omega} M\Big(\frac{u(x)}{\|u\|_{L^M(\Omega)}}\Big) \, dx \leq 1.
    \end{equation}
If function $M$ satisfies the $\Delta'$-condition, this inequality turns into equality.

It is well known (see, e.g. \cite{KR}) that if $M(u) = \frac{|u|^p}{p}$, $1 \leq p < \infty$, then 
    $$
        \|f\|_{L^{(M)}(\Omega)} = {p'}^{\frac{1}{p'}}\|f\|_{L^p(\Omega)}, \quad p'=\frac{p}{p-1}.
    $$

The analogue of the classical H\"older inequality in Orlicz spaces takes the following form:

\begin{prop}\label{HoldOrlprop}
    For any $f \in L^M(\Omega)$ and $g \in L^{M^*}(\Omega)$ the H\"older inequality holds:
    \begin{equation}\label{HoldOrl}
        \int_\Omega |f(x)g(x)|\, dx \leq 2\|f\|_{L^{M}(\Omega)} \|g\|_{L^{M^\ast}(\Omega)}.
    \end{equation}
\end{prop}

\subsection{Sobolev spaces and Poincar\'e inequalities}

Further we define spaces of weakly differentiable functions. We will use the Sobolev spaces based on Lebesgue spaces only. Detailed properties of these spaces can be found in \cite{M}.

The \textit{Sobolev space} $W^{1,p}(\Omega)$, $1\leq p\leq\infty$, is a set of all (equivalence classes of) locally integrable weakly differentiable functions $f: \Omega \to \mathbb{R}$, equipped with the following norm: 
    $$
        \|f\|_{W^{1,p}(\Omega)}=\|f\|_{L^p(\Omega)} + \|\nabla f\|_{L^p(\Omega)},
    $$
where $\nabla f$ is the weak gradient of the function $f$.

The \textit{seminormed Sobolev space} $L^{1,p}(\Omega)$, $1\leq p\leq\infty$, is a set of all (equivalence classes of) locally integrable weakly differentiable functions $f: \Omega \to \mathbb{R}$, equipped with the following seminorm: 
    $$
        \|f\|_{L^{1,p}(\Omega)}=\|\nabla f\|_{L^p(\Omega)}.
    $$

In the Sobolev spaces theory, capacity plays a crucial role as an outer measure associated with Sobolev spaces \cite{M}. In accordance to this approach, elements of Sobolev spaces $W^1_p(\Omega)$ are equivalence classes up to a set of $p$-capacity zero \cite{MH72}. 

A mapping $\varphi: \Omega \to \mathbb{R}^2$ belongs to the Sobolev space $W^{1,p}(\Omega,\mathbb{R}^2)$, if its coordinate functions belong to $W^{1,p}(\Omega)$. In this case, the formal Jacobi matrix $D\varphi(x)$ and its determinant (Jacobian) $J_\varphi$ are well defined at almost all points $x \in \Omega$. We use $|D\varphi(x)|$ to denote the operator norm of $D\varphi(x)$.

We will use the Poincar\'e inequality in two forms. A bounded domain $\Omega$ is said to be $(q,p)$\textit{-Poincar\'e domain}, if for all $f \in L^{1,p}(\Omega)$ the following equivalent inequalities hold. The first one better corresponds to the Min-Max Principle and states that for $f \in L^{1,p}(\Omega)$
    \begin{equation}\label{InfPoincare_pq}
        \inf\limits_{c \in \mathbb{R}}\|f - c\|_{L^q(\Omega)} \leq \widetilde{B}_{q,p}(\Omega)\|\nabla f\|_{L^p(\Omega)}.
    \end{equation}

The second one is more suitable for estimating the exact constant in the inequality and states that
    \begin{equation}\label{AvPoincare_pq}
        \|f - f_\Omega\|_{L^q(\Omega)} \leq B_{q,p}(\Omega)\|\nabla f\|_{L^p(\Omega)},
    \end{equation}
where $f_\Omega = |\Omega|^{-1}\int_\Omega f(x) \, dx$.

These inequalities are equivalent in the sense that (see, for example, \cite{HK00})
    \begin{equation}\label{EquivPoncare}
        \inf\limits_{c \in \mathbb{R}}\|f - c\|_{L^q(\Omega)} \leq \|f - f_{\Omega}\|_{L^q(\Omega)} \leq 2 \inf\limits_{c \in \mathbb{R}}\|f - c\|_{L^q(\Omega)}.
    \end{equation}

On the unit disk $\mathbb{D}(0,1) = \mathbb{D}$ there is the following estimate of the constant $B_{q,p}(\mathbb{D})$.

\begin{thm}\cite{GU16}\label{Poincare_pq}
    If $f \in W^{1,p}(\mathbb{D})$, $1 \leq p < \infty$, then for $0 \leq \kappa = 1/p - 1/q < 1/2$ the Poincar\'e-Sobolev inequality
    \begin{equation}\label{Pest}
        \left( \int_{\mathbb{D}} |f(y) - f_{\mathbb{D}}|^q \, dy \right)^{\frac{1}{q}} \leq B_{q,p}(\mathbb{D}) \left( \int_{\mathbb{D}} |\nabla f(y)|^p \, dy \right)^{\frac{1}{p}}
    \end{equation}
    holds with the constant
    $$
    B_{q,p}(\mathbb{D}) \leq \frac{2}{\pi^{\kappa}} \Big( \frac{1-\kappa}{1/2 - \kappa} \Big)^{1-\kappa}
    $$
\end{thm}

\subsection{Conformal mappings and composition operators}

Our main object is the first eigenvalue $\mu_\rho(\Omega)$ of the Laplace operator with density. As it was pointed out in the Introduction, it is natural to consider changes of variables generating by conformal mappings $\varphi: \mathbb{D} \to \Omega$. By the Riemann mapping theorem, for simply connected bounded domain $\Omega$ such mapping always exists, and hence we can use it for the approach based on the diagram \eqref{Diag}. 

Conformal mappings $\varphi: \mathbb{D} \to \Omega$ possess the Luzin $N$ and $N^{-1}$-properties and the following change of variable formulas hold:
\begin{equation}\label{CangeDir}
    \int_{\Omega} f(x) \, dx = \int_{\mathbb{D}} f(\varphi(y)) J_\varphi(y) \, dy,
\end{equation}
and for the inverse mapping $\varphi^{-1}: \Omega \to \mathbb{D}$:
\begin{equation}\label{CangeInv}
    \int_{\Omega} f(\varphi^{-1}(x)) J_{\varphi^{-1}}(x) \, dx = \int_{\mathbb{D}} f(y) \, dy,
\end{equation}

Therefore, Jacobians of the direct and inverse mappings are related in the following way
$$
    \int_{\Omega} \frac{1}{J_{\varphi^{-1}}(x)} \, dx = \int_{\Omega} J_{\varphi}(\varphi^{-1}(x)) \, dx.
$$

The first assumption on a domain $\Omega$ that we will use is the so-called $\alpha$-regularity of a domain. Let $\Omega \subset \mathbb{R}^2$ be a simply connected domain. We say that $\Omega$ is a \textit{conformal $\alpha$-regular domain}, $\alpha > 2$, if there exists a conformal mapping $\varphi: \mathbb{D} \to \Omega$ such that
$$
    \int_{\mathbb{D}} |D\varphi(x)|^{\alpha} \, dx = \int_{\mathbb{D}} (J_{\varphi}(x))^{\frac{\alpha}{2}} \, dx < \infty.
$$

This property was first introduced in \cite{BGU15}, see also \cite{GU16} and references therein. Note that the power $\alpha$ for this property does not depend on the choice of the conformal mapping, but depends on the hyperbolic geometry of $\Omega$ only.
 
The estimates of the Neumann Laplacian eigenvalues $\mu_\rho(\Omega)$, which are independent of the Jacobian $J_{\varphi}$, can be obtained with the help of the inverse H\"older inequality. This inequality was proved in \cite{GPU18} for the case, when $\Omega$ is a $K$-quasidisc.

A homeomorphism $\varphi: \Omega \to \Omega'$, $\Omega, \Omega' \subset \mathbb{R}^2$, is called \cite{A66} a $K$\textit{-quasiconformal} if it preserves the orientation, belongs to the Sobolev space $W^{1,2}_{\loc}(\Omega)$ and such that
$$
    |D \varphi(x)|^2 \leq K |J_{\varphi}(x)| \quad \text{a.e. in } \Omega.
$$

A domain $\Omega \subset \mathbb{R}^2$ is called a $K$\textit{-quasidisc} if it is an image of the unit disk $\mathbb{D}$ under a $K$-quasiconformal homeomorphism of the $\mathbb{R}^2$ onto itself. For such domains the following theorem holds.

\begin{thm}\cite{GPU18}\label{Jacobianest}
    Let $\Omega \subset \mathbb{R}^2$ be a $K$-quasidisc and $\varphi: \mathbb{D} \to \Omega$ be a conformal mapping. Suppose that $2 < \alpha < \frac{2K^2}{k^2-1}$. Then
    \begin{equation}\label{Jest}
        \left( \int_{\mathbb{D}} (J_{\varphi}(x))^{\frac{\alpha}{2}} \, dx \right)^{\frac{2}{\alpha}} \leq C_{J}(\alpha, K, |\Omega|),
    \end{equation}
    where
    $$
        C_{J}(\alpha, K, |\Omega|) = \frac{C^2_\alpha K^2 \pi^{\frac{2}{\alpha}-1}}{4} \exp\Big\{ \frac{K^2 \pi^2(2+\pi^4)^2}{2 \log 3} \Big\} |\Omega|,
    $$
    and
    $$
        C_\alpha = \frac{10^6}{[(\alpha-1)(1-\nu)]^{1/\alpha}}, \quad \nu = 10^{4\alpha}\frac{\alpha-2}{\alpha-1}(24 \pi^2K^2)^{\alpha}.
    $$
\end{thm}

Further we consider composition operators on Sobolev spaces. Composition operators on Lebesgue and Orlicz spaces can be defined in the same way.

Let $\Omega$ and $\Omega'$ be domains in $\mathbb{R}^2$. Then a homeomorphism $\varphi: \Omega' \to \Omega$ generates a bounded composition operator 
$$
    \varphi^{\ast}: L^{1,p}(\Omega) \to L^{1,q}(\Omega') \quad 1\leq q\leq p\leq\infty,
$$
by the composition rule $\varphi^{\ast}(f)=f\circ\varphi$, if for any function $f \in L^{1,p}(\Omega)$, the composition $\varphi^{\ast}(f)\in L^{1,q}(\Omega')$ is defined quasi-everywhere in $\Omega'$ and there exists a constant $K_{p,q}(\varphi;\Omega')<\infty$ such that 
$$
    \|\varphi^{\ast}(f)\|_{L^{1,q}(\Omega')} \leq K_{p,q}(\varphi;\Omega')\|f\|_{L^{1,p}(\Omega)}.
$$

It is well-known (e.g. \cite{GU16}), that conformal mappings $\varphi: \mathbb{D} \to \Omega$ generates an isometry of Sobolev spaces $L^{1,2}(\Omega)$ and $L^{1,2}(\mathbb{D})$. Therefore, for the conformal mapping $\varphi: \mathbb{D} \to \Omega$, the composition operator 
$$
    \varphi^{\ast}: L^{1,2}(\Omega) \to L^{1,2}(\mathbb{D})
$$
is bounded and the norm $\|\varphi^\ast\|=1$.

For the general case the following theorems are true (see \cite{U93, VU98} for the case of higher dimensions and more general classes of mappings).
\begin{thm}\cite{GU12}\label{CompTh} 
    A conformal mapping $\varphi: \Omega' \to \Omega$ generates a bounded composition operator 
    $$
        \varphi^{\ast}: L^{1,r}(\Omega) \to L^{1,p}(\Omega'), \quad 1 \leq p < r < \infty 
    $$
    if and only if  
    $$
        K_{r,p}(\varphi;\Omega') = \int_{\Omega'} |D\varphi(y)|^{\frac{(r-2)p}{r-p}} \, dy < \infty
    $$
    The norm of the operator $\varphi^\ast$ has the upper bound $\|\varphi^\ast\| \leq (K_{r,p}(\varphi;\Omega'))^{\frac{r-p}{rp}}$.
\end{thm}

\begin{thm}\cite{VU02, GU16}
    A conformal mapping $\psi: \Omega \to \Omega'$ generates a bounded composition operator 
    $$
        \psi^{\ast}: L^q(\Omega') \to L^s(\Omega), \quad 1 \leq s < q < \infty 
    $$
    if and only if  
    $$
        K_{r,s}(\psi; \Omega') = \Big( \int_{\Omega'} (J_{\psi^{-1}}(y))^{\frac{q}{q-s}} \, dy \Big)^{\frac{q-s}{qs}} < \infty.
    $$
    The norm of the operator $\varphi^\ast$ has the upper bound $\|\varphi^\ast\| = K_{r,s}(\psi; \Omega')$
\end{thm}

Together these two theorems give the following conditions guaranteeing the boundedness of the operator
$$
    \varphi^\ast: W^{1,r}(\Omega) \to W^{1,p}(\Omega').
$$

\begin{thm}\cite{GU16}
    Let $\Omega$ support $(q,p)$-Poincar\'e inequality for some $1 \leq p \leq q \leq \infty$ and a domain $\Omega'$ has finite measure. Suppose that a conformal mapping $\varphi: \Omega' \to \Omega$ generates a bounded composition operator 
    $$
        \varphi^{\ast}: L^{1,r}(\Omega) \to L^{1,p}(\Omega'), \quad 1 \leq p < r < \infty,
    $$
    and the inverse mapping $\varphi^{-1}: \Omega \to \Omega'$ generates a bounded composition operator 
    $$
        (\varphi^{-1})^{\ast}: L^q(\Omega') \to L^s(\Omega), \quad 1 \leq s < q < \infty 
    $$
    for some $r \leq s \leq q$.

    Then $\varphi: \Omega' \to \Omega$ generates a bounded composition operator 
    $$
         \varphi^\ast: W^{1,r}(\Omega) \to W^{1,p}(\Omega').
    $$
\end{thm}

In particular, for $r = 2$ and $\Omega' = \mathbb{D}$ we obtain the following corollary:
\begin{cor}\label{CompW}
    Let $\Omega$ be a simply connected bounded domain and $\varphi: \mathbb{D} \to \Omega$ be a conformal mapping. If $\Omega$ is $\alpha$-regular with $\alpha \in (\frac{p}{4},\infty]$, then $\varphi: \mathbb{D} \to \Omega$ generates a bounded composition operator 
    $$
        \varphi^\ast: W^{1,2}(\Omega) \to W^{1,p}(\mathbb{D}),  \quad 1 \leq p \leq 2.
    $$
\end{cor}

\section{From $(p,q)$-Laplacian to Laplacian with density}

In this section we establish a connection between the first eigenvalues of the Laplacian with density on the domain $\Omega$ with the quite general geometry and the homogeneous $(p,q)$-Laplacian on the unit disk. As it was pointed out in the Introduction, the approach is based on the composition operator on Sobolev and Lebesgue spaces and on estimates of the constant in the corresponding Poincar\'e inequality.

Let us consider the Neumann problem for the $(p,q)$-Laplacian on the unit disk:  

    \begin{align}\label{Neumann-pq}
        \begin{cases}
            -\Delta_p u = - \divg(|\nabla u|^{p-2} \nabla u)=\mu_{p,q} \|u\|^{p-q}_{L^q(\mathbb{D})} |u|^{q-2} u \,\, &\text{ in } \mathbb{D},\\
            \frac{\partial u}{\partial\nu}=0 &\text{ on } \partial\mathbb{D}.
        \end{cases}
    \end{align}

By the Min-Max Principle (see, e.g. \cite{GPU23}), the first non-zero eigenvalue of problem \eqref{Neumann-pq} can be found as
    $$
        \mu_{p,q}(\mathbb{D}) = \inf\left\{ \frac{\|\nabla u\|^p_{L^p(\mathbb{D})}}{\|u\|^p_{L^q(\mathbb{D})}} : u \in W^{1,p}(\mathbb{D})\setminus \{0\}, \int_{\mathbb{D}} |u|^{q-2}u = 0 \right\}.
    $$

This means that $(\mu_{p,q}(\mathbb{D}))^{-\frac{1}{p}}$ equals to the best constant in the $(q,p)$-Sobolev-Poincar\'e inequality
    $$
        \inf\limits_{c \in \mathbb{R}}\|f - c\|_{L^q(\mathbb{D})} \leq \widetilde{B}_{q,p}(\mathbb{D})\|\nabla f\|_{L^p(\mathbb{D})}.
    $$

Using inequalities \eqref{EquivPoncare}, we obtain the following estimates for $\mu_{p,q}(\mathbb{D})$:
    $$
        \frac{2^p}{B^p_{q,p}(\mathbb{D})} \geq \mu_{p,q} (\mathbb{D}) \geq \frac{1}{B^p_{q,p}(\mathbb{D})},
    $$
where $B_{q,p}(\mathbb{D})$ is a constant in Poincar\'e inequality \eqref{Pest}.

In the first simple observation, we establish a connection between $\mu_\rho(\Omega)$ and $\mu_{2,2}(\mathbb{D}) = \mu(\mathbb{D})$, i.e. the first eigenvalue of the usual Laplacian on the unit disc $\mathbb{D}$. Of course, this proposition can be obtain in a simpler way with the usual change of variables, but we propose the proof in this way to demonstrate the idea of the applied approach. 

\begin{thm}\label{22-est}
    Let $\Omega \subset \mathbb{R}^2$ be a simply connected bounded domain. Let also $\rho$ be a density. Consider a conformal mapping $\varphi: \mathbb{D} \to \Omega$. If 
    $$
        K(\Omega, \rho) = \esssup\limits_{x \in \Omega} \frac{\rho(x)} {J_{\varphi^{-1}}(x)} < \infty,
    $$
    then Sobolev-Poincar\'e inequality \eqref{Poincare_22} holds with the constant
    $$
        B^\rho_{2,2}(\Omega) \leq B_{2,2}(\mathbb{D}) \sqrt{K(\Omega, \rho)}.
    $$

    If, in addition, $\Omega$ is an $\alpha$-regular domain, then the embedding $W^{1,2}(\Omega) \hookrightarrow L^{1,2}(\Omega, \rho)$ is compact and
    
    $$
        \mu_{\rho}(\Omega) \geq \frac{\mu(\mathbb{D})}{K(\Omega, \rho)}.
    $$

    In the case $\rho(x) = J_{\varphi^{-1}}(x)$ a.e. on $\Omega$,
    $$
        \mu_{J_{\varphi^{-1}}}(\Omega) \geq \mu(\mathbb{D}).
    $$
\end{thm}

\begin{proof}
    As we want to compare the first non-trivial eigenvalue of the problem with density with usual Laplacian, we need to use the following anticommutative diagram:
    
    \begin{equation*}
        \begin{CD}
            L^{1,2}(\Omega) @>{\varphi^\ast}>> L^{1,2}(\mathbb{D}) \\
            @VVV @VVV \\
            L^2(\Omega, \rho) @<{(\varphi^{-1})^\ast}<< L^2(\mathbb{D})
        \end{CD}
    \end{equation*}

    This diagram explains us how to estimate the constant in the Poincar\'e inequality \eqref{Poincare_22}. Further, if, in addition, the operator
    $$
        \varphi^\ast : W^{1,2}(\Omega) \to W^{1,2}(\mathbb{D})
    $$
    is also bounded, we conclude that the embedding $W^{1,2}(\Omega) \hookrightarrow L^{1,2}(\Omega, \rho)$ is compact and we also obtain the estimates for the first eigenvalue with equality \eqref{MinMax}.
    
    By the above diagram, the constant $B^\rho_{2,2}(\Omega)$ can be estimated as follows:
    $$
        B^\rho_{2,2}(\Omega) \leq \|\varphi^\ast\| \cdot B_{2,2}(\mathbb{D}) \cdot \|(\varphi^{-1})^\ast\|.
    $$
    Indeed, in virtue of the definition of composition operators $\varphi^\ast$ and $(\varphi^{-1})^\ast$,
    \begin{align*}
        \inf\limits_{c \in \mathbb{R}}\|f-c\|_{L^2(\Omega, \rho)} &\leq \|(\varphi^{-1})^\ast\| \cdot \inf\limits_{c \in \mathbb{R}}\|f-c\|_{L^2(\mathbb{D})} \leq B_{2,2}(\mathbb{D}) \cdot \|(\varphi^{-1})^\ast\| \|f\|_{L^{1,2}(\mathbb{D})} \\
        &\leq \|\varphi^\ast\| \cdot B_{2,2}(\mathbb{D}) \cdot \|(\varphi^{-1})^\ast\| \|f\|_{L^{1,2}(\Omega)}.
    \end{align*}
    
    Since the conformal mapping induce an isomorphism of $L^{1,2}$-Sobolev spaces, the norm $\|\varphi^\ast\|$ equals to $1$. With the change of variable formula \eqref{CangeInv}, we can easily estimate the norm of $(\varphi^{-1})^\ast :  L^2(\mathbb{D}) \to L^2(\Omega, \rho)$.
    \begin{align*}
        \|f\circ \varphi^{-1}\|^2_{L^2(\Omega, \rho)} &= \int_{\Omega} |f(\varphi^{-1}(x))|^2 \rho(x) \, dx = \int_{\mathbb{D}} |f(y)|^2 \frac{\rho(\varphi(y))} {J_{\varphi^{-1}}(\varphi(y))} \, dy \\
        & \leq \esssup\limits_{y \in \mathbb{D}} \frac{\rho(\varphi(y))} {J_{\varphi^{-1}}(\varphi(y))} \|f\|^2_{L^2(\mathbb{D})}.
    \end{align*}

    The above inequalities imply that $\|(\varphi^{-1})^\ast\| \leq \esssup\limits_{x \in \Omega} \left( \frac{\rho(x)} {J_{\varphi^{-1}}(x)} \right)^{\frac{1}{2}}$. Therefore,
    $$
        B^\rho_{2,2}(\Omega) \leq B_{2,2}(\mathbb{D}) \esssup\limits_{x \in \Omega} \left( \frac{\rho(x)} {J_{\varphi^{-1}}(x)} \right)^{\frac{1}{2}}.
    $$

    If $\Omega$ is an $\alpha$-regular domain, then by Corollary \ref{CompW} the composition operator 
    $$
        \varphi^\ast : W^{1,2}(\Omega) \to W^{1,2}(\mathbb{D})
    $$
    is also bounded and with the help of \eqref{MinMax} we have the estimate for the first eigenvalue
    $$
        \mu_{\rho}(\Omega) = \frac{1}{(B^{\rho}_{2,2}(\Omega))^2} \geq \frac{1}{(B_{2,2}(\mathbb{D}))^2 \|(\varphi^{-1})^\ast\|} \geq \frac{\mu(\mathbb{D})}{\esssup\limits_{x \in \Omega} \frac{\rho(x)} {J_{\varphi^{-1}}(x)}}.
    $$

    In particular, if $\rho(x) = J_{\varphi^{-1}}(x)$ a.e. on $\Omega$,
    $$
        \mu_{J_{\varphi^{-1}}}(\Omega) \geq \mu(\mathbb{D}).
    $$
\end{proof}

Further, we will proceed with the case of different $p$ and $q$. The following theorem provides a lower estimate for $\mu_\rho(\Omega)$ based on the connection with homogeneous $(p,q)$-Laplacian. This estimate also contains the term depending on the density $\rho$ and the Jacobian of the conformal mapping $\varphi^{-1}: \Omega \to \mathbb{D}$.

\begin{thm}\label{pq-est}
    Let $\Omega \subset \mathbb{R}^2$ be a simply connected bounded domain. Let also $\rho$ be a density and $1 \leq p < 2$, $2 < q \leq \frac{2p}{2-p}$. Consider a conformal mapping $\varphi: \mathbb{D} \to \Omega$. If
    $$
        K_q(\Omega, \rho, \varphi) = \Big\| \frac{\rho}{J^{2/q}_{\varphi^{-1}}} \Big\|_{L^{\frac{q}{q-2}}(\Omega)} < \infty,
    $$
    then Sobolev-Poincar\'e inequality \eqref{Poincare_22} holds with the constant 
    $$
        B^\rho_{2,2}(\Omega) \leq \pi^{\frac{2-p}{p}} B_{q,p}(\mathbb{D}) \sqrt{K_q(\Omega, \rho, \varphi)}.
    $$

    If, in addition, $\Omega$ is an $\alpha$-regular domain with $\alpha \in (\frac{4}{p}, \infty]$, and $q < \frac{2p}{2-p}$, then the embedding 
    $W^{1,2}(\Omega) \hookrightarrow L^2(\Omega, \rho)$ is compact and 
    $$
        \mu_{\rho}(\Omega) \geq \frac{\mu_{p,q}(\mathbb{D})}{2^{2p}\pi^{\frac{2(2-p)}{p}} K_q(\Omega, \rho, \varphi)}.
    $$

    In the case $\rho(x) = J^{2/q}_{\varphi^{-1}}(x)$ a.e. on $\Omega$,
    $$
        \mu_{J_{\varphi^{-1}}}(\Omega) \geq \frac{\mu_{p,q}(\mathbb{D})}{2^{2p}\pi^{\frac{2(2-p)}{p}}}.
    $$
\end{thm}

\begin{proof}
    In this theorem we estimate the first non-trivial eigenvalue of the Laplacian with density via the eigenvalue of the $(p,q)$-Laplacian. Therefore, the anticommutative diagram takes the following form:

    \begin{equation*}
        \begin{CD}
            L^{1,2}(\Omega) @>{\varphi^\ast}>> L^{1,p}(\mathbb{D}) \\
            @VVV @VVV \\
            L^2(\Omega, \rho) @<{(\varphi^{-1})^\ast}<< L^q(\mathbb{D})
        \end{CD}
    \end{equation*}

    Again, by this diagram we can estimate the constant in the Poincar\'e inequality \eqref{Poincare_22}. If, in addition, operator
    $$
        \varphi^\ast : W^{1,2}(\Omega) \to W^{1,p}(\mathbb{D})
    $$
    is bounded and the embedding $W^{1,p}(\mathbb{D}) \hookrightarrow L^q(\mathbb{D})$ is compact, we conclude that the embedding $W^{1,2}(\Omega) \hookrightarrow L^{1,2}(\Omega, \rho)$ is also compact and we obtain the estimates for the first eigenvalue with equality \eqref{MinMax}.
    
    As in Theorem \ref{22-est}, by the above diagram we have the following chain of inequalities:
    \begin{align*}
        \inf\limits_{c \in \mathbb{R}}\|f-c\|_{L^2(\Omega, \rho)} &\leq \|(\varphi^{-1})^\ast\| \cdot \inf\limits_{c \in \mathbb{R}}\|f-c\|_{L^q(\mathbb{D})} \leq \|(\varphi^{-1})^\ast\| \cdot \|f-f_{\mathbb{D}}\|_{L^q(\mathbb{D})}  \\
        &\leq B_{q,p}(\mathbb{D}) \cdot \|(\varphi^{-1})^\ast\| \|f\|_{L^{1,p}(\mathbb{D})} \leq \|\varphi^\ast\| \cdot B_{q,p}(\mathbb{D}) \cdot \|(\varphi^{-1})^\ast\| \|f\|_{L^{1,2}(\Omega)}.
    \end{align*}
    
    This gives us an estimate for $B^\rho_{2,2}(\Omega)$:
    $$
        B^\rho_{2,2}(\Omega) \leq \|\varphi^\ast\| \cdot B_{q,p}(\mathbb{D}) \cdot \|(\varphi^{-1})^\ast\|.
    $$

    Theorem \ref{CompTh} provides the estimate for the norm of $\varphi^\ast: L^{1,2}(\Omega) \to L^{1,p}(\mathbb{D})$:
    $$
        \|\varphi^\ast\| \leq \left(\int_{\mathbb{D}}  |D\varphi(y)|^{\frac{(2-2)p}{2-p}} \, dy \right)^{\frac{2-p}{2p}} = \pi^{\frac{2-p}{p}}.
    $$

    Further we estimate the norm of $(\varphi^{-1})^\ast: L^q(\mathbb{D}) \to L^2(\Omega, \rho)$ with the help of H\"older's inequality and the change of variable formula.

    \begin{align*}
        \|f\circ \varphi^{-1}\|_{L^2(\Omega, \rho)} &= \left(\int_{\Omega} |f(\varphi^{-1}(x))|^2 \rho(x) \, dx \right)^{\frac{1}{2}} = \left( \int_{\mathbb{D}} |f(y)|^2 \frac{\rho(\varphi(y))} {J_{\varphi^{-1}}(\varphi(y))} \, dy \right)^{\frac{1}{2}} \\
        & \leq \left(\int_{\mathbb{D}} |f(y)|^q \, dy \right)^{\frac{1}{q}} \left(\int_{\mathbb{D}} \Big(\frac{\rho(\varphi(y))} {J_{\varphi^{-1}}(\varphi(y))}\Big)^{\frac{q}{q-2}} \, dy \right)^{\frac{q-2}{2q}} \\
        &= \|f\|_{L^q(\mathbb{D})} \left(\int_{\Omega} \Big(\frac{\rho(x)} {J_{\varphi^{-1}}(x)}\Big)^{\frac{q}{q-2}} J_{\varphi^{-1}}(x) \, dx \right)^{\frac{q-2}{2q}} = \|f\|_{L^q(\mathbb{D})} \Big\| \frac{\rho}{J_{\varphi^{-1}}^{2/q}} \Big\|^{\frac{1}{2}}_{L^{\frac{q}{q-2}}(\Omega)}.
    \end{align*}

    Hence we obtain the estimate for the norm $\|(\varphi^{-1})^\ast\|$:
    $$
        \|(\varphi^{-1})^\ast\| \leq \Big\| \frac{\rho}{J_{\varphi^{-1}}^{2/q}} \Big\|^{\frac{1}{2}}_{L^{\frac{q}{q-2}}(\Omega)}.
    $$

    Finally, this implies that
    $$
        B^\rho_{2,2}(\Omega) \leq \pi^{\frac{2-p}{p}} B_{q,p}(\mathbb{D}) \Big\| \frac{\rho}{J_{\varphi^{-1}}^{2/q}} \Big\|^{\frac{1}{2}}_{L^{\frac{q}{q-2}}(\Omega)}.
    $$

    Now, by Corollary \ref{CompW}, if $\Omega$ is an $\alpha$-regular domain with $\alpha \in (\frac{4}{p}, \infty]$, the composition operator
    $$
        \varphi^\ast : W^{1,2}(\Omega) \to W^{1,p}(\mathbb{D})
    $$
    is also bounded. It is well-known, that the Sobolev embedding $W^{1,p}(\mathbb{D}) \hookrightarrow L^q(\mathbb{D})$ is compact, if $q < \frac{2p}{2-p}$. Then, under this additional conditions, the embedding $W^{1,2}(\Omega) \hookrightarrow L^{1,2}(\Omega, \rho)$ is also compact. By \eqref{MinMax} we conclude that 
    $$
        \mu_{\rho}(\Omega) \geq \frac{1}{\pi^{\frac{2(2-p)}{p}} B^2_{q,p}(\mathbb{D}) \Big\| \frac{\rho}{J^{2/q}_{\varphi^{-1}}} \Big\|_{L^{\frac{q}{q-2}}(\Omega)}} \geq \frac{\mu_{p,q}(\mathbb{D})}{2^{2p}\pi^{\frac{2(2-p)}{p}} \Big\| \frac{\rho}{J^{2/q}_{\varphi^{-1}}} \Big\|_{L^{\frac{q}{q-2}}(\Omega)}},
    $$
    and, in the case $\rho(x) = J^{2/q}_{\varphi^{-1}}(x)$ a.e. on $\Omega$,
    $$
        \mu_{J_{\varphi^{-1}}}(\Omega) \geq \frac{\mu_{p,q}(\mathbb{D})}{2^{2p}\pi^{\frac{2(2-p)}{p}}}.
    $$
\end{proof}

If $\Omega$ is a $K$-quasidisc and $\frac{2q}{q-2} < \alpha < \frac{2K^2}{k^2-1}$, with the help of Theorem \ref{Jacobianest} and Theorem \ref{Poincare_pq} we can obtain an estimate for $\mu_{\rho}(\Omega)$ independent of the conformal mapping $\varphi$ and prove the following corollary of Theorem \ref{pq-est}:

\begin{cor}\label{indest}
    Let $\Omega \subset \mathbb{R}^2$ be a $K$-quasidisc and $\frac{2q}{q-2} < \alpha < \frac{2K^2}{K^2-1}$, $1 \leq p < 2$, $2 < q < \frac{2p}{2-p}$ and $\rho$ be a density. Then
    \begin{equation}\label{rhoest}
        \mu_{\rho}(\Omega) \geq \frac{\Big(\frac{1-\kappa}{1/2-\kappa}\Big)^{2\kappa-2}}{4\pi^{\frac{2(p-2)}{p} - 2\kappa}\|\rho\|^{\frac{q-2}{q}}_{L^{\frac{q(\alpha-2)}{q\alpha - 2q -2\alpha}}(\Omega)} (C_{J}(\alpha, K, |\Omega|))^{\frac{2\alpha}{q(\alpha - 2)}}}.
    \end{equation}
    where $C_{J}(\alpha, K, |\Omega|)$ as in \eqref{Jest}, and $\kappa = 1/p-1/q$.
\end{cor}

\begin{proof}
    This corollary is just a direct consequence of the H\"older inequality. Indeed, if $\alpha > \frac{2q}{q-2}$, then $\frac{(\alpha-2)(q-2)}{4}>1$ and we can apply the H\"older inequality in the following way:
    \begin{align*}
        \Big\| \frac{\rho}{J_{\varphi^{-1}}^{2/q}} \Big\|_{L^{\frac{q}{q-2}}(\Omega)} &= \left(\int_{\Omega} \frac{(\rho(x))^{\frac{q}{q-2}}}{(J_{\varphi^{-1}}(x))^{\frac{2}{q-2}}}\, dx \right)^{\frac{q-2}{q}} =  \left(\int_{\Omega} (\rho(x))^{\frac{q}{q-2}}(J_{\varphi}(\varphi^{-1}(x)))^{\frac{2}{q-2}}\, dx \right)^{\frac{q-2}{q}} \\
        &\leq \left( \int_\Omega (J_{\varphi}(\varphi^{-1}(x)))^{\frac{\alpha-2}{2}} \, dx \right)^{\frac{4}{q(\alpha-2)}} \left( \int_\Omega \rho(x)^{\frac{q(\alpha-2)}{q\alpha - 2q -2\alpha}} \, dx \right)^{\frac{q\alpha - 2q -2\alpha}{q(\alpha-2)}\frac{q-2}{q}} \\
        &=  \left( \int_{\mathbb{D}} (J_{\varphi}(y))^{\frac{\alpha}{2}} \, dy \right)^{\frac{2}{\alpha} \frac{2\alpha}{q(\alpha - 2)}} \|\rho\|^{\frac{q-2}{q}}_{L^{\frac{q(\alpha-2)}{q\alpha - 2q -2\alpha}}(\Omega)}.
    \end{align*}

    Note that the condition $\alpha > \frac{2q}{q-2}$ implies $\alpha > \frac{4}{p}$, and hence we are in conditions of Theorem \ref{pq-est} and the embedding $W^{1,2}(\Omega) \hookrightarrow L^2(\Omega, \rho)$ is compact.

    Inequality \eqref{Jest} gives us the following estimate:
    $$
        \Big\| \frac{\rho}{J_{\varphi^{-1}}^{2/q}} \Big\|^{\frac{1}{2}}_{L^{\frac{q}{q-2}}(\Omega)} \leq \|\rho\|^{\frac{q-2}{q}}_{L^{\frac{q(\alpha-2)}{q\alpha - 2q -2\alpha}}(\Omega)} (C_{J}(\alpha, K, |\Omega|))^{\frac{2\alpha}{q(\alpha - 2)}}.
    $$

    Here $\kappa = \frac{1}{p}-\frac{1}{q}<\frac{1}{2}$, if $q < \frac{2p}{2-p}$. Then, combine the above inequality with \eqref{Pest},
    $$
        \mu_{\rho}(\Omega) \geq \frac{\Big(\frac{1-\kappa}{1/2-\kappa}\Big)^{2\kappa-2}}{4\pi^{\frac{2(p-2)}{p} - 2\kappa}\|\rho\|^{\frac{q-2}{q}}_{L^{\frac{q(\alpha-2)}{q\alpha - 2q -2\alpha}}(\Omega)} (C_{J}(\alpha, K, |\Omega|))^{\frac{2\alpha}{q(\alpha - 2)}}}.
    $$
\end{proof}  

The constant $C_J$ in inequality \eqref{rhoest} is very large and hence almost useless for obtaining sharp estimates of $\mu_\rho(\Omega)$. But at the same time, it shows that the first non-zero eigenvalue necessarily depends on the density $\rho$ and, moreover, by changing the density, we can make $\mu_{\rho}(\Omega)$ as big as we wish. Namely, the following theorem holds:

\begin{cor}
    Let $\Omega \subset \mathbb{R}^2$ be a $K$-quasidisc endowed with the sequence of Gaussian densities $\{\rho_n\}_{n \in \mathbb{N}} = \{e^{-n|x|^2}\}_{n \in \mathbb{N}}$. Let also $\frac{2q}{q-2} < \alpha < \frac{2K^2}{K^2-1}$, $1 \leq p < 2$, $2 < q < \frac{2p}{2-p}$. Then
    $$
        \mu_{\rho_n}(\Omega) \geq C(\alpha, p, q, K, |\Omega|) n^{\frac{q\alpha - 2q -2\alpha}{q(\alpha-2)}\frac{q-2}{q}}.
    $$
    In particular, if $K$ tends to $1$ ($\alpha \to \infty$), $\lambda_{\rho_n}(\Omega)$ grows linearly with $n$.
\end{cor}

\begin{proof}
    Denote $s=\frac{q(\alpha-2)}{q\alpha - 2q -2\alpha}$. Then the statement follows from simple calculations:
    $$
        \|\rho_n\|^{\frac{q-2}{q}}_{L^s(\Omega)} = \left( \int_\Omega e^{-ns|x|^2} \, dx \right)^{\frac{q-2}{qs}} \leq \left( \int_{\mathbb{R}^2} e^{-ns|x|^2} \, dx \right)^{\frac{q-2}{qs}} = \left(\frac{\pi}{ns}\right)^{\frac{q-2}{qs}}.
    $$
    
\end{proof}

\section{Laplacian with density and the sharp Sobolev embedding}

The constant in the estimate of $\mu_{pq}(\Omega)$ with Corollary \ref{indest} is far from optimal. But this shows that the grater power $q$ we take, the better estimate we get. Therefore, it is natural to replace the $L^{1,p} \to L^q$ embedding with the optimal embedding for the borderline case $p=n=2$. This section is devoted to proving of refined versions of Theorem \ref{pq-est} and Corollary \ref{indest} in the sense of the optimal Sobolev embedding. It is known (see, e.g. \cite{P65, T67, J61}) that the optimal target space for $L^{1,2}(\Omega)$ is the Orlicz space $L^M(\Omega)$, where $M$ is a Young function $M(u) = e^{u^2}-1$. This means that for $\Omega$ regular enough we have the following embedding:
\begin{equation}\label{SobOrl}
    L^{1,2}(\Omega) \to L^M(\Omega). 
\end{equation}

It was also proved in \cite{T67} that for all Young functions $M_\varepsilon \prec M$ the embedding is compact
\begin{equation}\label{SobOrlComp}
    L^{1,2}(\Omega) \hookrightarrow L^{M_\varepsilon}(\Omega). 
\end{equation}

The embedding \eqref{SobOrl} equivalents to the corresponding Poincar\'e inequality:
\begin{equation}\label{PoinOrl}
    \|f-\med(f)\|_{L^M(\Omega)} \leq B_{M,2} \|\nabla f\|_{L^2(\Omega)},
\end{equation}
where $\med(f)$ is a median of $u$ and is defined as
$$
    \med(f) = \inf\{t \geq 0: |\{x \in \Omega: f(x)>t\}| \leq \frac{|\Omega|}{2}\}
$$

Remind that the $p$-isocapacitary function (see, for example, \cite{M}) $\nu_{\Omega,p}: [0; \frac{|\Omega|}{2}] \to [0, \infty)$ of $\Omega \subset \mathbb{R}^2$ is given by
$$
    \nu_{\Omega,p} = \inf\Big\{ \cp_p(E): E \subset \Omega, s \leq |E| \leq \frac{|\Omega|}{2} \Big\}.
$$

The following theorem is a particular case of Theorem A from \cite{CM23} (see also \cite{T67}) and gives the sharp estimate of the constant in the inequality \eqref{PoinOrl}.

\begin{thm}\cite{CM23}
    Let $\Omega$ be a domain in $\mathbb{R}^2$ and $|\Omega| < \infty$. Then the embedding 
    $$
        L^{1,2}(\Omega) \to L^M(\Omega)
    $$
    holds with the Young function $M$, such that $\frac{M(u)}{u^n}$ is non-decreasing, if and only if 
    \begin{equation}\label{capcond}
        \Big \| \frac{1}{\nu_{\Omega, 2}(s) (M^{-1}(1/s))^2 } \Big \|_{L^\infty(0; \frac{|\Omega|}{2})} \leq \infty.
    \end{equation}

    Moreover, the inequality \eqref{PoinOrl} holds with
    $$
        B_{M,2}(\Omega) = c \Big \| \frac{1}{\nu_{\Omega, 2}(s) (M^{-1}(1/s))^2 } \Big \|^{\frac{1}{2}}_{L^\infty(0; \frac{|\Omega|}{2})},
    $$
    where $c$ is a multiplicative constant, independent on $\Omega$.
\end{thm}

Obviously, the function $M(u) = e^{u^2}-1$ satisfies the condition of the above theorem. At the same time, all the functions $M_\varepsilon \prec M$, that grows faster then $u^n$, also fit into the above theorem, for example $M_\varepsilon(u) = e^{u^{\frac{2}{\varepsilon}}}-1$, $\varepsilon>1$. Therefore, together with \cite{T67} this gives the continuous embedding \eqref{SobOrl} and compact embedding \eqref{SobOrlComp} for Lipschitz domains $\Omega$. 

Now we are ready to formulate our main result that allows to estimate the first non-trivial Neumann eigenvalue $\mu_\rho(\Omega)$ for the Laplace operator with density $\rho$ on the domain that is a conformal image of the unit disc $\mathbb{D}$.

\begin{thm}\label{OrlEstThm}
Let $\Omega \subset \mathbb{R}^2$ be a simply connected bounded domain. Let also $\rho$ be a density and Young functions $M$ and $\Phi$ be defined as $M(u) = e^{u^2}-1$ and $\Phi(u) = u\log(u+e)$. Consider a conformal mapping $\varphi: \mathbb{D} \to \Omega$. If
    $$
        K_{\Phi}(\Omega, \rho, \varphi)  = \Big\|\frac{\rho}{J_{\varphi^{-1}}\Phi^{-1}(\frac{1}{J_{\varphi^{-1}}})}\Big\|_{L^\Phi(\Omega)} < \infty,
    $$
    then Sobolev-Poincar\'e inequality \eqref{Poincare_22} holds with the constant
    $$
        B^\rho_{2,2}(\Omega) \leq 3\sqrt{2} B_{M,2}(\mathbb{D}) \sqrt{K_{\Phi}(\Omega, \rho, \varphi)}.
    $$
    
    In particular, if $\rho \equiv 1$ then
    \begin{equation}\label{OrlEstUniform}
        B^\rho_{2,2}(\Omega) = B_{2,2}(\Omega) \leq \sqrt{2} B_{M,2}(\mathbb{D})\|J_\varphi\|^{\frac{1}{2}}_{L^{\Phi}(\mathbb{D})}.
    \end{equation}
    
    If, in addition, $\Omega$ is an $\alpha$-regular domain with $\alpha \in (2, \infty]$, and $M_\varepsilon \prec M$ ($\Phi \prec \Phi_\varepsilon$), then the embedding 
    $W^{1,2}(\Omega) \hookrightarrow L^2(\Omega, \rho)$ is compact and 
    $$
        \mu_{\rho}(\Omega) \geq \frac{1}{18 (B_{M_\varepsilon,2}(\mathbb{D}))^2 K_{\Phi_\varepsilon}(\Omega, \rho, \varphi)}.
    $$

    In the case $\rho(x) = J_{\varphi^{-1}}(x)\Phi_\varepsilon^{-1}(\frac{1}{J_{\varphi^{-1}}(x)})$ a.e. on $\Omega$,
    $$
        \mu_{\rho}(\Omega) \geq \frac{1}{18 (B_{M_\varepsilon,2}(\mathbb{D}))^2}.
    $$
\end{thm}

\begin{proof}

    The anticommutative diagram in this case takes the following form:

    \begin{equation*}
        \begin{CD}
            L^{1,2}(\Omega) @>{\varphi^\ast}>> L^{1,2}(\mathbb{D}) \\
            @VVV @VVV \\
            L^2(\Omega, \rho) @<{(\varphi^{-1})^\ast}<< L^M(\mathbb{D})
        \end{CD}
    \end{equation*}

    Here we again first prove the estimate for the constant in the Poincar\'e inequality \eqref{Poincare_22} and then, under the additional assumptions, we provide the estimate for $\mu_\rho(\Omega)$.

    According to the diagram, we have the following chain of inequalities:
    \begin{align*}
        \inf\limits_{c \in \mathbb{R}}\|f-c\|_{L^2(\Omega, \rho)} &\leq \|(\varphi^{-1})^\ast\| \cdot \inf\limits_{c \in \mathbb{R}}\|f-c\|_{L^M(\mathbb{D})} \leq \|(\varphi^{-1})^\ast\| \cdot \|f-\med(f)\|_{L^M(\mathbb{D})}  \\
        &\leq B_{M,2}(\mathbb{D}) \cdot \|(\varphi^{-1})^\ast\| \|f\|_{L^{1,2}(\mathbb{D})} \leq \|\varphi^\ast\| \cdot B_{M,2}(\mathbb{D}) \cdot \|(\varphi^{-1})^\ast\| \|f\|_{L^{1,2}(\Omega)}.
    \end{align*}
    
    This gives us an estimate for $B^\rho_{2,2}(\Omega)$:
    $$
        B^\rho_{2,2}(\Omega) \leq \|\varphi^\ast\| \cdot B_{M,2}(\mathbb{D}) \cdot \|(\varphi^{-1})^\ast\|.
    $$

    In this case, $\varphi^\ast: L^{1,2}(\Omega) \to L^{1,2}(\mathbb{D})$ is an isometry and hence $\|\varphi^\ast\|=1$. Let us estimate the norm $(\varphi^{-1})^\ast: L^M(\mathbb{D}) \to L^2(\Omega, \rho)$. Using the change of variable formula \eqref{CangeDir} and the H\"older inequality \eqref{HoldOrl} for the conjugate pair $\widetilde\Phi(u) = (1 + u)\log(1+u) - u$ and $\widetilde\Phi^\ast(u) = e^u - 1$, we obtain
    \begin{align*}
        \|f \circ \varphi^{-1}\|_{L^2(\Omega, \rho)} &= \left( \int_{\Omega} |f(\varphi^{-1}(x))|^2 \rho(x) \, dx \right)^{\frac{1}{2}} = \left( \int_{\mathbb{D}} |f(y)|^2 \rho(\varphi(y))J_{\varphi}(y) \, dy \right)^{\frac{1}{2}} \\
        &\leq \sqrt{2} \|f^2\|^{\frac{1}{2}}_{L^{\widetilde\Phi^\ast}(\mathbb{D})} \| J_{\varphi} \cdot \rho \circ \varphi \|^{\frac{1}{2}}_{L^{\widetilde\Phi}(\mathbb{D})}.
    \end{align*}

    By the definition of the Luxembourg norm, taking $\lambda  = \widetilde\lambda^2$,
    \begin{align*}
        \|f^2\|_{L^{\widetilde\Phi^\ast}(\mathbb{D})} &= \inf\Big\{\lambda: \int_{\mathbb{D}} \exp\Big(\frac{|f(y)|^2}{\lambda}\Big) - 1 \, dy \leq 1 \Big\} \\
        &= \Big(\inf\Big\{{\widetilde\lambda}: \int_{\mathbb{D}} \exp\Big(\frac{|f(y)|^2}{{\widetilde\lambda}^2}\Big) - 1 \, dy \leq 1 \Big\}\Big)^2 = \|f\|^2_{L^M(\mathbb{D})}.
    \end{align*}

    The Young function $\widetilde\Phi(u) = (1 + u)\log(1+u) - u$ satisfies the $\Delta'$-condition, but not for all $u>0$ (see \cite{KR}). We can change this function to the equivalent one that satisfies the $\Delta'$-condition for all $u>0$, namely, to the function $\Phi(u) = u\log(u+e)$. Indeed,
    \begin{multline*}
        \Phi(uv) = uv\log(uv+e) \leq uv\log((u+e)(v+e)) \\
        \leq 2u\log(u+e)v\log(v+e)=2\Phi(u)\Phi(v),
    \end{multline*}
    which means that $\Phi$ satisfies $\Delta'$-condition for all $u \geq 0$ with $C_{\Delta'}=2$, and
    $$
        \lim\limits_{u \to \infty} \frac{(1 + u)\log(1+u) - u}{u\log(u+e)} = 1,
    $$
    that implies $\Phi \sim \widetilde\Phi$.
    
    By \cite{KR}, the space $L^{\widetilde\Phi}(\mathbb{D})$ and $L^{\Phi}(\mathbb{D})$ coincide as sets. Moreover, $\widetilde\Phi(u) \leq \Phi(u)$ for all $u$. Hence for all functions $f$, $\|f\|_{L^{\widetilde\Phi}(\mathbb{D})} \leq \|f\|_{L^{\Phi}(\mathbb{D})}$. Therefore, we can conclude that
    \begin{multline*}
        \|f \circ \varphi^{-1}\|_{L^2(\Omega, \rho)} \leq \sqrt{2} \|f^2\|^{\frac{1}{2}}_{L^{\widetilde\Phi^\ast}(\mathbb{D})} \| J_{\varphi} \cdot \rho \circ \varphi \|^{\frac{1}{2}}_{L^{\widetilde\Phi}(\mathbb{D})} \\
        = \sqrt{2}\|f\|_{L^M(\mathbb{D})} \| J_{\varphi} \cdot \rho \circ \varphi \|^{\frac{1}{2}}_{L^{\widetilde\Phi}(\mathbb{D})} \leq \sqrt{2}\|f\|_{L^M(\mathbb{D})} \| J_{\varphi} \cdot \rho \circ \varphi \|^{\frac{1}{2}}_{L^{\Phi}(\mathbb{D})}.
    \end{multline*}

    The above inequality implies
    $$
        \|(\varphi^{-1})^\ast\| \leq \sqrt{2}\| J_{\varphi} \cdot \rho \circ \varphi \|^{\frac{1}{2}}_{L^{\Phi}(\mathbb{D})},
    $$
    and hence, the following estimate for the constant $B^\rho_{2,2}(\Omega)$ holds:
    $$
        B^{\rho}_{2,2}(\Omega) \leq \sqrt{2} B_{M,2}(\mathbb{D}) \| J_{\varphi} \cdot \rho \circ \varphi \|^{\frac{1}{2}}_{L^{\Phi}(\mathbb{D})}.
    $$
    In the case $\rho \equiv 1$, we obtain \eqref{OrlEstUniform}:
    $$
        B_{2,2}(\Omega) \leq \sqrt{2} B_{M,2}(\mathbb{D}) \| J_{\varphi} \|^{\frac{1}{2}}_{L^{\Phi}(\mathbb{D})}.
    $$

    To prove inequality \eqref{OrlEstDens}, we need to change the variable in the term $\| J_{\varphi} \cdot \rho \circ \varphi \|_{L^{\Phi}(\mathbb{D})}$. For this purpose we changed $\widetilde\Phi$ on $\Phi$ to use the global $\Delta'$-condition. Let us consider a functional
    $$
        \int_{\mathbb{D}} \left|\frac{\rho(\varphi(y))}{\lambda}\frac{J_\varphi(y)}{\Phi^{-1}(J_\varphi(y))} \Phi^{-1}(J_\varphi(y)) g(y)\right| \, dy,
    $$
    where $\lambda$ is a positive constant and $g$ is an arbitrary function from the space $L^{\Phi^{\ast}}(\mathbb{D})$ with the property $\int_{\mathbb{D}} \Phi^\ast(g(y)) \, dy \leq 1$.

    By the Young inequality, global $\Delta'$-condition for $\Phi$ and the change of variable formula,
    \begin{align*}
        \int_{\mathbb{D}} \Big| \frac{\rho(\varphi(y))}{\lambda} &\frac{J_\varphi(y)}{\Phi^{-1}(J_\varphi(y))} \Phi^{-1}(J_\varphi(y)) g(y) \Big| \, dy \\
        &\leq \int_{\mathbb{D}} \Phi\left( \frac{\rho(\varphi(y))}{\lambda}\frac{J_\varphi(y)}{\Phi^{-1}(J_\varphi(y))} \Phi^{-1}(J_\varphi(y)) \right) \, dy + \int_{\mathbb{D}} \Phi^\ast(g(y)) \, dy \\
        &\leq 2 \int_{\mathbb{D}} \Phi\left( \frac{\rho(\varphi(y))}{\lambda}\frac{J_\varphi(y)}{\Phi^{-1}(J_\varphi(y))} \right) J_\varphi(y) \, dy + \int_{\mathbb{D}} \Phi^\ast(g(y)) \, dy \\
        &= 2 \int_{\Omega} \Phi\left( \frac{\rho(x)}{\lambda}\frac{1}{J_{\varphi^{-1}}(x)\Phi^{-1}(\frac{1}{J_{\varphi^{-1}}(x)})} \right) \, dx + \int_{\mathbb{D}} \Phi^\ast(g(y)) \, dy
    \end{align*}

    Taking $\lambda=\Big\|\frac{\rho}{J_{\varphi^{-1}}\Phi^{-1}(\frac{1}{J_{\varphi^{-1}}})}\Big\|_{L^\Phi(\Omega)}$, we infer
    \begin{align*}
        \| J_{\varphi} \cdot &\rho \circ \varphi \|_{L^{\Phi}(\mathbb{D})} \leq \lambda \Big\| \frac{J_{\varphi} \cdot \rho \circ \varphi}{\lambda}  \Big\|_{L^{(\Phi)}(\mathbb{D})} \\
        & = \lambda \sup\left\{\int_{\mathbb{D}} \left|\frac{\rho(\varphi(y))}{\lambda}\frac{J_\varphi(y)}{\Phi^{-1}(J_\varphi(y))} \Phi^{-1}(J_\varphi(y)) g(y)\right| \, dy \right\} \\
        & \leq \lambda \sup\left\{2 \int_{\Omega} \Phi\left( \frac{\rho(x)}{\lambda}\frac{1}{J_{\varphi^{-1}}(x)\Phi^{-1}(\frac{1}{J_{\varphi^{-1}}(x)})} \right) \, dx + \int_{\mathbb{D}} \Phi^\ast(g(y)) \, dy  \right\} \\
        & \leq 3\Big\|\frac{\rho}{J_{\varphi^{-1}}\Phi^{-1}(\frac{1}{J_{\varphi^{-1}}})}\Big\|_{L^\Phi(\Omega)},
    \end{align*}
    where the last inequality holds due to the property \eqref{NormProp} of the Luxembourg norm and due to considering functions $g$ from the unit ball of the space $L^{\Phi^{\ast}}(\mathbb{D})$.

    Finally, we obtain the estimate
    $$
        B^\rho_{2,2}(\Omega) \leq 2\sqrt{3} B_{M,2}(\mathbb{D}) \Big\|\frac{\rho}{J_{\varphi^{-1}}\Phi^{-1}(\frac{1}{J_{\varphi^{-1}}})}\Big\|^{\frac{1}{2}}_{L^\Phi(\Omega)}.
    $$

    Now we want to conclude, that under the assumptions of the theorem, the embedding $W^{1,2} \hookrightarrow L^{2}(\Omega, \rho)$ is compact. The $\alpha$-regularity of $\Omega$ gives us the boundedness of the operator
    $$
        \varphi^\ast: W^{1,2}(\Omega) \to W^{1,2}(\mathbb{D}).
    $$
    Taking $M_\varepsilon \prec M$, we get the compactness of the embedding $W^{1,2}(\mathbb{D}) \hookrightarrow L^{M_\varepsilon}(\mathbb{D})$. The simplest example of such function is a function $M_\varepsilon(u) = e^{u^{\frac{2}{\varepsilon}}}-1$, $\varepsilon > 1$. It is easy to check (see, e.g. Section 6.3 in \cite{KR}), that in this case the function $\Phi_\varepsilon \sim u\log^\varepsilon u$ as dual to the $\widetilde{\Phi}^\ast_\varepsilon = e^{u^{\frac{1}{\varepsilon}}}-1$.
    
    Hence, by the equality \eqref{MinMax}, we obtain the estimate \eqref{OrlEstDens}:
    $$
        \mu_{\rho}(\Omega) \geq \frac{1}{18 (B_{M_\varepsilon,2}(\mathbb{D}))^2 \Big\|\frac{\rho}{J_{\varphi^{-1}}\Phi_\varepsilon^{-1}(\frac{1}{J_{\varphi^{-1}}})}\Big\|_{L^\Phi_\varepsilon(\Omega)}}.
    $$

    In the case $\rho(x) = J_{\varphi^{-1}}(x)\Phi_\varepsilon^{-1}(\frac{1}{J_{\varphi^{-1}}(x)})$ a.e. on $\Omega$,
    $$
        \mu_{\rho}(\Omega) \geq \frac{1}{18 (B_{M_\varepsilon,2}(\mathbb{D}))^2}.
    $$    
\end{proof}

The following theorem is a corollary of Theorem \ref{OrlEstThm} and is a refined version of Corollary \ref{indest}. It is sharp in the sense of optimality of the space $L^M(\mathbb{D})$ as a target space for the Sobolev embedding.

\begin{thm}\label{OrlEstDens}
    Let $\Omega \subset \mathbb{R}^2$ be a $K$-quasidisc and $2 < \alpha < \frac{2K^2}{K^2-1}$, and $\rho$ be a density. Let also 
    $$
        \Phi(u) = u \log(u+e), \quad \Psi(u) = \frac{2}{\alpha} \left(\Phi^{-1}(u)(e^{\Phi^{-1}(u)}-e)\right)^{\frac{\alpha-2}{2}}.
    $$
    Then
    \begin{equation}\label{rhoestOrlicz}
        B^{\rho}_{2,2} \leq B_{M,2}(\mathbb{D})\Big(\frac{\widetilde{C}_J(\alpha, K, |\Omega|)}{\Phi^{-1}(1/\|\Phi(\rho)\|_{L^{\Psi^\ast}(\Omega)})}\Big)^{\frac{1}{2}},
    \end{equation}
    where 
    $$
        \widetilde{C}_J(\alpha, K, |\Omega|) = \frac{288 C_\Psi}{\Phi^{-1}\left(\frac{1}{\Psi((\frac{\alpha}{\alpha-2})^{\frac{\alpha-2}{2}} C^{\frac{\alpha}{2}}_J(\alpha, K, |\Omega|))}\right)},
    $$
    the constant $C_{J}(\alpha, K, |\Omega|)$ as in \eqref{Jest}, and $C_\Psi$ is a constant from the $\nabla'$-condition for $\Psi$.
\end{thm}

\begin{proof}
    Similarly to the proof of Corollary \ref{indest}, the main ingredient here is the H\"older inequality \eqref{HoldOrl}, but the technical part is more involved. Let us consider the term
    $$
        \Big\|\frac{\rho}{J_{\varphi^{-1}}\Phi^{-1}(\frac{1}{J_{\varphi^{-1}}})}\Big\|_{L^\Phi(\Omega)}
    $$
    from \eqref{OrlEstDens}. Since $\varphi: \mathbb{D} \to \Omega$ is a conformal mapping, we can rewrite it in the following form
    $$
        \Big\|\rho \cdot \frac{J_{\varphi} \circ \varphi^{-1}}{\Phi^{-1}(J_{\varphi} \circ \varphi^{-1})}\Big\|_{L^\Phi(\Omega)}.
    $$

    Let us consider the corresponding functional
    \begin{equation}\label{FuncDens}
        \int_{\Omega} \left|\frac{\rho(x)}{\lambda}\frac{J_\varphi(\varphi^{-1}(x))}{\Phi^{-1}(J_\varphi(\varphi^{-1}(x)))} g(x)\right| \, dx,
    \end{equation}
    where $\lambda$ is a positive constant and $g$ is an arbitrary function from the space $L^{\Phi^{\ast}}(\Omega)$ with the property $\int_{\Omega} \Phi^\ast(g(x)) \, dx \leq 1$.

    Using the Young inequality \eqref{YoungIneq} and the $\Delta'$-condition for function $\Phi$, for $\lambda = \lambda_1 \lambda_2 > 0$, we come to the following estimate
    \begin{multline*}
        \int_{\Omega} \left|\frac{\rho(x)}{\lambda}\frac{J_\varphi(\varphi^{-1}(x))}{\Phi^{-1}(J_\varphi(\varphi^{-1}(x)))} g(x)\right| \, dx \\
        \leq 2\int_{\Omega} \Phi\left(\frac{\rho(x)}{\lambda_1}\right) \Phi\left( \frac{J_\varphi(\varphi^{-1}(x))}{\lambda_2\Phi^{-1}(J_\varphi(\varphi^{-1}(x)))} \right) \, dx + \int_{\Omega} \Phi^\ast(g(x)) \, dx \\
        \leq 8 \Phi(\lambda_1^{-1})\Phi(\lambda_2^{-1}) \int_{\Omega} \Phi(\rho(x)) \Phi\left( \frac{J_\varphi(\varphi^{-1}(x))}{\Phi^{-1}(J_\varphi(\varphi^{-1}(x)))} \right) \, dx + 1.
    \end{multline*}

    The second term in the above inequality is bounded by $1$. For the first term we want to apply the H\"older inequality \eqref{HoldOrl} in such a way that gives us a norm of $J_\varphi$ in the space $L^{\frac{\alpha}{2}}(\mathbb{D})$. The first step is to determine a Young function $\Psi$ with the property
    $$
        \Psi \left(\Phi\left(\frac{u}{\Phi^{-1}(u)}\right)\right) = \frac{2 u^{\frac{\alpha-2}{2}}}{\alpha}.
    $$

    Changing the variable and using $\Phi(u) = u\log(u+e)$,
    $$
        \Phi\left(\frac{u}{\Phi^{-1}(u)}\right) = t \quad \Rightarrow \quad u = \Phi^{-1}(t)(e^{\Phi^{-1}(t)}-e),
    $$
    we obtain
    $$
        \Psi(t) = \frac{2}{\alpha} \left(\Phi^{-1}(t)(e^{\Phi^{-1}(t)}-e)\right)^{\frac{\alpha-2}{2}}.
    $$

    It's not hard to check that $\Psi: [0,\infty) \to [0, \infty]$ is a Young function. It is continuous as a composition of continuous functions, convexity follows because of the term $e^t-e$, and the limits at $0$ and $\infty $ are $0$ and $\infty$. Moreover, this function satisfy the $\nabla'$-condition
    $$
        \Psi(u)\Psi(v) \leq \Psi(C_{\Psi} uv) \quad \text{ for all } u,v \geq 0.
    $$
    This follows from the $\Delta'$-condition for $\Phi$ (and hence, the reverse inequality for the inverse function), and from \cite[Theorem 6.7]{KR}, as the conjugate function for $e^u-e$ satisfy the $\Delta'$-condition.

    Therefore, we can use the H\"older inequality \eqref{HoldOrl} with functions $\Psi$ and $\Psi^\ast$. It gives us
    $$
        \int_{\Omega} \Phi(\rho(x)) \Phi\Big( \frac{J_\varphi(\varphi^{-1}(x))}{\Phi^{-1}(J_\varphi(\varphi^{-1}(x)))} \Big) \, dx \leq 2\|\Phi(\rho)\|_{L^{\Psi^\ast}(\Omega)} \Big\| \Phi\Big( \frac{J_{\varphi} \circ \varphi^{-1}}{\Phi^{-1}(J_{\varphi}\circ \varphi^{-1})} \Big) \Big\|_{L^{\Psi}(\Omega)}.
    $$

    It remains to show that we can estimate last term by $\|J_\varphi\|_{L^{\frac{\alpha}{2}}(\mathbb{D})}$. For some $\delta > 0$, by the $\nabla'$-condition for $\Psi$ and the change of variable formula,
    \begin{multline*}
        \frac{\Psi(C^{-1}_\Psi\delta)}{\Psi(C^{-1}_\Psi\delta)} \int_{\Omega} \Psi\left(\frac{\Phi\left( \frac{J_{\varphi}(\varphi^{-1}(x))}{\Phi^{-1}(J_{\varphi}(\varphi^{-1}(x)))} \right)}{\delta}\right) \, dx  \\
        \leq \frac{1}{\Psi(C^{-1}_\Psi\delta)} \int_{\Omega} \Psi\left(\Phi\left( \frac{J_{\varphi}(\varphi^{-1}(x))}{\Phi^{-1}(J_{\varphi}(\varphi^{-1}(x)))} \right)\right) \, dx \\
        = \frac{1}{\Psi(C^{-1}_\Psi\delta)}\int_{\Omega} \frac{2}{\alpha} (J_{\varphi}(\varphi^{-1}(x)))^{\frac{\alpha-2}{2}} \, dx = \int_{\mathbb{D}} \frac{2}{\alpha} \left(\frac{J_{\varphi}(y)}{(\Psi(C^{-1}_\Psi\delta))^{\frac{2}{\alpha}}}\right)^{\frac{\alpha}{2}} \, dy.
    \end{multline*}

    If $\delta = C_\Psi \Psi^{-1}((\frac{\alpha}{\alpha-2})^{\frac{\alpha-2}{2}}\|J_{\varphi}\|^{\frac{\alpha}{2}}_{L^{\frac{\alpha}{2}}(\mathbb{D})})$, then the last integral in the above expression equals to $1$. Hence,
    \begin{multline*}
        \left\| \Phi\left( \frac{J_{\varphi} \circ \varphi^{-1}}{\Phi^{-1}(J_{\varphi}\circ \varphi^{-1})} \right) \right\|_{L^{\Psi}(\Omega)} = \\
        \inf\Big\{ \delta: \int_{\Omega} \Psi\left(\frac{\Phi\left( \frac{J_{\varphi}(\varphi^{-1}(x))}{\Phi^{-1}(J_{\varphi}(\varphi^{-1}(x)))} \right)}{\delta}\right) \, dx \leq 1 \Big\} \\
        \leq  C_\Psi \Psi^{-1}((\frac{\alpha}{\alpha-2})^{\frac{\alpha-2}{2}}\|J_{\varphi}\|^{\frac{\alpha}{2}}_{L^{\frac{\alpha}{2}}(\mathbb{D})}).
    \end{multline*}

    Finally, returning to functional \eqref{FuncDens}, for $\lambda = \lambda_1 \lambda_2 > 0$ we conclude
    \begin{multline*}
        \int_{\Omega} \left|\frac{\rho(x)}{\lambda}\frac{J_\varphi(\varphi^{-1}(x))}{\Phi^{-1}(J_\varphi(\varphi^{-1}(x)))} g(x)\right| \, dx \\
        \leq 8 \Phi(\lambda_1^{-1})\Phi(\lambda_2^{-1}) \int_{\Omega} \Phi(\rho(x)) \Phi\left( \frac{J_\varphi(\varphi^{-1}(x))}{\Phi^{-1}(J_\varphi(\varphi^{-1}(x)))} \right) \, dx + 1 \\
        \leq 16 \Phi(\lambda_1^{-1})\Phi(\lambda_2^{-1}) \|\Phi(\rho)\|_{L^{\Psi^\ast}(\Omega)} \left\| \Phi\left( \frac{J_{\varphi} \circ \varphi^{-1}}{\Phi^{-1}(J_{\varphi}\circ \varphi^{-1})} \right) \right\|_{L^{\Psi}(\Omega)} \\
        \leq 16 \Phi(\lambda_1^{-1})\Phi(\lambda_2^{-1}) \|\Phi(\rho)\|_{L^{\Psi^\ast}(\Omega)} C_\Psi \Psi^{-1}((\frac{\alpha}{\alpha-2})^{\frac{\alpha-2}{2}}\|J_{\varphi}\|^{\frac{\alpha}{2}}_{L^{\frac{\alpha}{2}}(\mathbb{D})}). 
    \end{multline*}

    Taking 
    $$
        \lambda_1 = \frac{1}{\Phi^{-1}(1/\|\Phi(\rho)\|_{L^{\Psi^\ast}(\Omega)})}, \quad \lambda_2 = \frac{1}{\Phi^{-1}(1/\Psi((\frac{\alpha}{\alpha-2})^{\frac{\alpha-2}{2}}\|J_{\varphi}\|^{\frac{\alpha}{2}}_{L^{\frac{\alpha}{2}}(\mathbb{D})}))},
    $$
    we obtain
    \begin{multline*}
        \Big\|\frac{\rho}{J_{\varphi^{-1}}\Phi^{-1}(\frac{1}{J_{\varphi^{-1}}})}\Big\|_{L^\Phi(\Omega)} \leq \lambda  \Big\|\rho \cdot \frac{J_{\varphi} \circ \varphi^{-1}}{\lambda\Phi^{-1}(J_{\varphi} \circ \varphi^{-1})}\Big\|_{L^{(\Phi)}(\Omega)} \\
        = \lambda \sup\left\{\int_{\Omega} \left|\frac{\rho(x)}{\lambda}\frac{J_\varphi(\varphi^{-1}(x))}{\Phi^{-1}(J_\varphi(\varphi^{-1}(x)))} g(x)\right| \, dx : \int_{\Omega} \Phi^{\ast}(g(x))\, dx \leq 1 \right\} \\
        \leq 16 \lambda_1\lambda_2 \Phi(\lambda_1^{-1})\Phi(\lambda_2^{-1}) \|\Phi(\rho)\|_{L^{\Psi^\ast}(\Omega)} C_\Psi \Psi^{-1}((\frac{\alpha}{\alpha-2})^{\frac{\alpha-2}{2}}\|J_{\varphi}\|^{\frac{\alpha}{2}}_{L^{\frac{\alpha}{2}}(\mathbb{D})})\\
        \leq 16 C_\Psi \frac{1}{\Phi^{-1}(1/\|\Phi(\rho)\|_{L^{\Psi^\ast}(\Omega)})} \frac{1}{\Phi^{-1}(1/\Psi((\frac{\alpha}{\alpha-2})^{\frac{\alpha-2}{2}}\|J_{\varphi}\|^{\frac{\alpha}{2}}_{L^{\frac{\alpha}{2}}(\mathbb{D})}))}.
    \end{multline*}

    Therefore, for 
    $$
        \widetilde{C}_J(\alpha, K, |\Omega|) = \frac{288 C_\Psi}{\Phi^{-1}\left(\frac{1}{\Psi((\frac{\alpha}{\alpha-2})^{\frac{\alpha-2}{2}} C^{\frac{\alpha}{2}}_J(\alpha, K, |\Omega|))}\right)},
    $$
    where $C_J(\alpha, K, |\Omega|)$ is taken from \eqref{Jest},
    $$
        B^{\rho}_{2,2} \leq B_{M,2}(\mathbb{D})\Big(\frac{\widetilde{C}_J(\alpha, K, |\Omega|)}{\Phi^{-1}(1/\|\Phi(\rho)\|_{L^{\Psi^\ast}(\Omega)})}\Big)^{\frac{1}{2}}.
    $$
\end{proof}

Considering Young functions  
    $$
        \Phi_\varepsilon(u) = u \log^{\varepsilon}(u+e), \quad \Psi_{\varepsilon}(u) = \frac{2}{\alpha} \left(\Phi_\varepsilon^{-1}(u)(e^{(\Phi_\varepsilon^{-1}(u))^{\frac{1}{\varepsilon}}}-e)\right)^{\frac{\alpha-2}{2}}.
    $$
we come to the compact case and it makes possible to use the Min-Max Principle and equality \eqref{MinMax}. This gives us the estimate 
\begin{equation}\label{rhoestOrliczMinMax}
        \mu_{\rho}(\Omega) \geq \frac{1}{\widetilde{C}_J(\alpha, K, |\Omega|) (B_{M_\varepsilon,2}(\mathbb{D}))^2 \frac{1}{\Phi_\varepsilon^{-1}(1/\|\Phi_\varepsilon(\rho)\|_{L^{\Psi_\varepsilon^\ast}(\Omega)})}},
\end{equation}

\begin{rem}
    It is possible to obtain estimates of the type \eqref{rhoestOrlicz} using the another estimates for the norm of the Jacobian $J_\varphi$. The estimate of $\|J\|_{L^\Phi}$ with $\Phi(u) = u \log^p(u+e)$, can be found, for example in \cite[Theorem 6.1]{HK14}. The planar case was considered in \cite{AGSR10} (see also \cite{HK14} for further links). Unfortunately, there are some technical obstacles to use these results. 
\end{rem}

\textbf{Acknowledgements.} The author is grateful to Alexander Ukhlov for the introduction to the topic and for the fruitful discussions.

\vskip 0.3cm

Alexander Menovschikov; Department of Mathematics, Ben-Gurion University of the Negev, P.O.Box 653, Beer Sheva, 8410501, Israel
 
\emph{E-mail address:} \email{menovschikovmath@gmail.com} \\

\end{document}